\documentclass[11pt,english]{amsart}
\usepackage[T1]{fontenc}
\usepackage[latin9]{inputenc}
\usepackage{amsthm}
\usepackage{amstext}
\usepackage{esint}

\makeatletter
\usepackage{amsmath,amsfonts,amssymb,amsthm,epsfig}

\usepackage{color}
\usepackage[final,allcolors=blue,colorlinks=true]{hyperref}

\def\R{\mathbb R}

\def\al{\alpha}

\def\de{\delta}

\def\la{\lambda}

\def\na{\nabla}
\def\Ga{\Gamma}  
\def\Om{\Omega}  
\def\De{\Delta}      

\def\wq{\infty}
\def\pa{\partial}

\newcommand{\medint}{-\kern -,375cm\int}         
\newcommand{\medintinrigo}{-\kern -,315cm\int}
\newcommand\esssup{\text{\rm \,esssup\,}}

\numberwithin{equation}{section}
\newtheorem{theorem}{Theorem}[section]

\newtheorem{corollary}[theorem]{Corollary}

\newtheorem{lemma}[theorem]{Lemma}

\newtheorem{proposition}[theorem]{Proposition}

\theoremstyle{definition}

\makeatother

\usepackage{babel}
\begin{document}
\title[Liouville problems]{Nonlinear Liouville problems in a quarter plane} 

                                \author[C.-L. Xiang]{Chang-Lin Xiang}         

\address[]{University of Jyvaskyla, Department of Mathematics and Statistics, P.O. Box 35, FI-40014 University of Jyvaskyla, Finland.}
\email[]{Xiang\_math@126.com}

\begin{abstract}
We answer affirmatively the open problem proposed by Cabr\'e and Tan in their  paper "Positive solutions of nonlinear problems involving the square root of the Laplacian" (see Adv. Math. {\bf 224} (2010), no. 5, 2052-2093).
\end{abstract}

\maketitle

{
\small    
\keywords {\noindent {\bf Keywords:}  Nonexistence; Monotonicity; The method of moving spheres}
\smallskip
\newline
\subjclass{\noindent {\bf 2010 Mathematics Subject Classification: 35B09; 35B53; 35J60}  }
\tableofcontents
}
\bigskip

\section{Introduction and main result}

In this paper, we consider positive solutions of the nonlinear boundary
value problem
\begin{equation}
\begin{cases}
\De u=0, & \text{in }\R_{++}^{n+1},\\
u(x,y)>0, & \text{in }\R_{++}^{n+1},\\
u(0,y)=0 & \text{on }\{x_{n}=0,y\ge0\},\\
{\displaystyle \frac{\pa u}{\pa\nu}}=u^{p} & \text{on }\{x_{n}>0,y=0\},
\end{cases}\label{eq: Cabre-Tan-nD}
\end{equation}
where $n\ge1$, $1\le p<\wq$, $\R_{++}^{n+1}=\left\{ (x_{1},x_{2},\ldots,x_{n},y)\in\R^{n+1}:x_{n}>0,y>0\right\} $
and $\nu$ is the unit out normal to $\R_{++}^{n+1}$ at $\{x_{n}>0,y=0\}$. 

Problem (\ref{eq: Cabre-Tan-nD}) was probably studied first by Cabr\'e
and Tan \cite{Cabre-Tan-2010}. The motivation comes from the study
of the Gidas-Spruck \cite{Gidas-Spruck-1981} type apriori estimates
for solutions of the nonlinear nonlocal problem 
\begin{equation}
\begin{cases}
A_{1/2}u=u^{p} & \text{in }\Om,\\
u>0 & \text{in }\Om,\\
u=0 & \text{on }\pa\Om,
\end{cases}\label{eq: local nonlocal prob.}
\end{equation}
where $\Om\subset\R^{n}$ is a bounded smooth domain and $A_{1/2}$
is the square root of the Laplacian operator $-\De$ in $\Om$ with
zero Dirichlet boundary values on $\pa\Om$. For the precise definition
of $A_{1/2}$, we refer the readers to Cabr\'e and Tan \cite{Cabre-Tan-2010}.
Problem (\ref{eq: Cabre-Tan-nD}) appears as one of the two limiting
equations when applying the method of blow-up to solutions of Eq.
(\ref{eq: local nonlocal prob.}); the other related limiting equation
is given by 
\begin{equation}
\begin{cases}
\De u=0, & \text{in }\R_{+}^{n+1},\\
u(x,y)>0, & \text{in }\R_{+}^{n+1},\\
{\displaystyle \frac{\pa u}{\pa\nu}}=u^{p} & \text{on }\pa\R_{+}^{n+1}.
\end{cases}\label{eq: Related limiting problem}
\end{equation}
It is well known that Eq. (\ref{eq: Related limiting problem}) has
no weak solutions for all $p<(n+1)/(n-1)$ when $n\ge2$ (see e.g.
\cite{Hu-1994,Li-Zhang-2003,Li-Zhu-1995,Ou-1996}). For related Liouville
type problems in the whole space $\R^{n+1}$, we refer to e.g. Caffarelli
et al. \cite{Caffarelli et al. 1989}, Chen, Li and Ou \cite{Chen-Li-1991,Chen-Li-Ou-2006}
and Y.Y. Li \cite{LiYY-2004}. 

By the regularity theory developed in Cabr\'e and Tan \cite{Cabre-Tan-2010},
solutions of Eq. (\ref{eq: Cabre-Tan-nD}) in the weak sense are shown
to be classical in the sense that, any weak solution of Eq. (\ref{eq: Cabre-Tan-nD})
belongs to $C^{2}(\R_{++}^{n+1})\cap C^{1}(\overline{\R_{++}^{n+1}})$.
Thus, we restrict our attention to classical solutions of Eq. (\ref{eq: Cabre-Tan-nD}).
As one of their main results, Cabr\'e and Tan \cite{Cabre-Tan-2010}
obtained the following result (see \cite[Theorem 1.5]{Cabre-Tan-2010}). 

\begin{theorem} \label{thm: Cabre-Tan-2010} Let $n\ge2$ and $1<p\le(n+1)/(n-1)$.
Then, there exists no bounded classical solution to Eq. (\ref{eq: Cabre-Tan-nD}).

Equivalently, there exists no bounded solution of equation 
\[
\begin{cases}
A_{1/2}u=u^{p} & \text{in }\R_{+}^{n},\\
u>0 & \text{in }\R_{+}^{n},\\
u=0 & \text{on }\pa\R_{+}^{n},
\end{cases}
\]
where $A_{1/2}$ is the square root of the Laplacian in $\R_{+}^{n}=\{x_{n}>0\}$
with zero Dirichlet boundary conditions on $\pa\R_{+}^{n}$. \end{theorem} 

We briefly review the approach of Cabr\'e and Tan \cite{Cabre-Tan-2010}
in below for later use. Let $n\ge2$ and $1<p\le(n+1)/(n-1)$. Suppose
that $u$ is a classical solution to Eq. (\ref{eq: Cabre-Tan-nD}).
First Cabr\'e and Tan \cite{Cabre-Tan-2010} derived the symmetry
of $u$ with respect to $x_{i}$, $1\le i\le n-1$, by combining the
Kelvin transform and the method of moving planes. Since Eq. (\ref{eq: Cabre-Tan-nD})
is translation invariant with respect to $x_{i}$, $1\le i\le n-1$,
it follows that $u$ depends only on $x_{n}$ and $y$ (see \cite[Proposition 6.3]{Cabre-Tan-2010}).
Hence, Eq. (\ref{eq: Cabre-Tan-nD}) is reduced to the following problem
in the two dimensional quarter plane 
\begin{equation}
\begin{cases}
\De u=0, & \text{in }\R_{++}^{2}=\{x>0,y>0\},\\
u>0, & \text{in }\R_{++}^{2},\\
u=0 & \text{on }\{x=0,y\ge0\},\\
{\displaystyle \frac{\pa u}{\pa\nu}}=u^{p} & \text{on }\{x>0,y=0\}.
\end{cases}\label{eq: Cabre-Tan-1D}
\end{equation}
Then they proved that Eq. (\ref{eq: Cabre-Tan-1D}) has no bounded
classical solution by applying a Hamiltonian identity for the half-Laplacian
found by Cabr\'e and Sol\`{a}-Morales \cite{Cabre-Sola-Morales-2005}.
In this way, Theorem \ref{thm: Cabre-Tan-2010} is proved. 

Some remarks are in order. First, we remark that to reduce Eq. (\ref{eq: Cabre-Tan-nD})
to Eq. (\ref{eq: Cabre-Tan-1D}), the boundedness assumption of the
solution is not needed. Thus the boundedness assumption is only used
when deriving the nonexistence of solutions of Eq. (\ref{eq: Cabre-Tan-1D}).
Next, we remark that, under the boundedness assumption of the solution,
Cabr\'e and Tan \cite{Cabre-Tan-2010} derived nonexistence results
for equations of type (\ref{eq: Cabre-Tan-1D}) under far more general
boundary conditions (see Cabr\'e and Tan \cite[Proposition 6.4]{Cabre-Tan-2010}).
However, they pointed out that Theorem \ref{thm: Cabre-Tan-2010}
is open without the assumption of boundedness of the solution. 

In this paper, we remove their boundedness assumption. The following
theorem is our main result. 

\begin{theorem} \label{thm: Main reuslt} Let $n\ge1$. Assume that
$1\le p\le(n+1)/(n-1)$ for $n\ge2$ and $1\le p<\wq$ for $n=1$.
Then, there exists no classical solution to Eq. (\ref{eq: Cabre-Tan-nD}).
\end{theorem}

We prove Theorem \ref{thm: Main reuslt} in Section \ref{sec: Proof-of-main result}.
The idea is as follows. First note that by the symmetry result of
Cabr\'e and Tan \cite[Proposition 6.3]{Cabre-Tan-2010}, Eq. (\ref{eq: Cabre-Tan-nD})
is reduced to Eq. (\ref{eq: Cabre-Tan-1D}). Then, as a key gradient,
we show that any positive solution of Eq. (\ref{eq: Cabre-Tan-1D})
is monotone increasing in the $x$-direction. This idea is inspired
by the work of Li and Lin \cite{Li-Lin-2012}, where a nonlinear elliptic
PDE with two Sobolev-Hardy critical exponents are considered. Finally,
combining the monotonicity result together with the very general result
of Cabr\'e and Tan \cite[ Proposition 6.2]{Cabre-Tan-2010} (see
Proposition \ref{prop: Proposition of 6.2} below), we obtain Theorem
\ref{thm: Main reuslt}. To complete the proof of Theorem \ref{thm: Main reuslt},
we will give some necessary results in the next Section \ref{sec: Some-necessary-results}.
In the last section, we give an extension of Theorem \ref{thm: Main reuslt},
which can be seen as an analogue of Cabr\'e and Tan \cite[Proposition 6.4]{Cabre-Tan-2010}.

Our notations are standard. $B_{R}(x)$ is the open ball in $\R^{N}$
centered at $x$ with radius $R>0$. Whenever $E\subset\R^{N}$ is
a Lebesgue measurable set, we denote by $|E|$ the $N$-dimensional
Lebesgue measure of set $E$. Let $\Om$ be an arbitrary domain in
$\R^{N}$. For any $1\le s\le\infty$, $L^{s}(\Om)$ is the Banach
space of Lebesgue measurable functions $u$ such that the norm 
\[
\|u\|_{s,\Om}=\begin{cases}
\left(\int_{\Om}|u|^{s}\right)^{\frac{1}{s}} & \text{if }1\le s<\infty\\
\esssup_{\Om}|u| & \text{if }s=\infty
\end{cases}
\]
is finite. A function $u$ belongs to the Sobolev space $W^{1,s}(\Om)$
if $u\in L^{s}(\Om)$ and its first order weak partial derivatives
also belong to $L^{s}(\Om)$. For the properties of the Sobolev functions,
we refer to the monograph \cite{Ziemer}.

\section{Some preliminaries\label{sec: Some-necessary-results}}

In this section we collect some useful results for later use. The
first one concerns with Sobolev-Poincar\'e type inequalities in planar
domains, which will be used in the proof of Lemma \ref{lem: monotonicity}.

\begin{lemma} \label{lem: Sobolev-Poincare inequality}Let $\Om\subset\R_{+}^{2}$
be a bounded Lipschitz domain with a partial boundary $\Ga\subset\pa\R_{+}^{2}$
($\Ga$ could be an empty set). Then for any number $q$, $1\le q<\wq$,
there exists a constant $C_{q}$, depending only on $q$, such that
the following inequality holds 
\[
\|u\|_{q,\Om}\le C_{q}|\Om|^{\frac{1}{q}}\|\na u\|_{2,\Om}
\]
for all functions $u\in W^{1,2}(\Om)\cap C(\Om\cup\Ga)$ with $u=0$
on $\pa\Om\cap\R_{+}^{2}$. \end{lemma}
\begin{proof}
This lemma may be well known to specialist. We give a sketch of proof
for the reader's convenience. 

First consider the case $\Ga=\emptyset$. In this case, Lemma \ref{lem: Sobolev-Poincare inequality}
is a direct consequence of the Trudinger-Moser inequality (see \cite{Moser-1967,Pohozaev-1965,Trudinger-1980})
\[
\sup_{\|\na u\|_{2,\Om}=1}\int_{\Om}e^{\al|u|^{2}}\le C_{\al}|\Om|,
\]
where $\al\le4\pi$ and $C_{\al}>0$ is a constant depending only
on $\al$. Take $\al=1$. We obtain that 
\[
\|u\|_{2k,\Om}\le C_{k}|\Om|^{\frac{1}{2k}}\|\na u\|_{2,\Om}
\]
for all $k\in\{1,2.\ldots\}$. Now, Lemma \ref{lem: Sobolev-Poincare inequality}
follows easily from above and H\"older's inequality in the case $\Ga=\emptyset$. 

In the general case when $\Ga\neq\emptyset$, it suffices to consider
the even extension 
\[
\tilde{u}(x,y)=\begin{cases}
u(x,y) & \text{for }y\ge0\\
u(x,-y) & \text{for }y<0
\end{cases}
\]
 for $u\in W^{1,2}(\Om)\cap C(\Om\cup\Ga)$ with $u=0$ on $\pa\Om\cap\R_{+}^{2}$.
Then this case is reduced to the previous one. The proof of Lemma
\ref{lem: Sobolev-Poincare inequality} is finished. 
\end{proof}
The next very general result is Proposition 6.2 of Cabr\'e and Tan
\cite{Cabre-Tan-2010} (see also Chipot et al. \cite{Chipot et al-1998}),
which will be used in the proof of Theorem \ref{thm: Main reuslt}.

\begin{proposition} \label{prop: Proposition of 6.2}Suppose that
$v$ weakly solves 
\[
\begin{cases}
-\De v\ge0 & \text{in }\R_{+}^{2},\\
v\ge0 & \text{in }\R_{+}^{2},\\
{\displaystyle \frac{\pa v}{\pa\nu}}\ge0 & \text{on }\pa\R_{+}^{2}.
\end{cases}
\]
Then $v$ is a constant. \end{proposition}

\section{Proof of Theorem \ref{thm: Main reuslt}\label{sec: Proof-of-main result}}

In this section we prove Theorem \ref{thm: Main reuslt}. As already
reviewed the approach of Cabr\'e and Tan \cite{Cabre-Tan-2010} in
the introduction part, to prove Theorem \ref{thm: Main reuslt}, we
only need to prove that Eq. (\ref{eq: Cabre-Tan-1D}) has no classical
solution. We use the following lemma as a key gradient of the proof.

\begin{lemma} \label{lem: monotonicity}Suppose that $u$ is a classical
positive solution to Eq. (\ref{eq: Cabre-Tan-1D}). Then $u_{x}(x,y)>0$
for all $(x,y)\in\overline{\R_{++}^{2}}$. \end{lemma} 

Before giving a proof of Lemma \ref{lem: monotonicity}, we will apply
Lemma \ref{lem: monotonicity} to prove Theorem \ref{thm: Main reuslt}.

\begin{proof}[Proof of Theorem \ref{thm: Main reuslt}] Suppose that
$u$ is a positive solution to Eq. (\ref{eq: Cabre-Tan-1D}). Define
the odd extension $\bar{u}:\R_{+}^{2}\to\R$ of $u$ by 
\[
\bar{u}(x,y)=\begin{cases}
u(x,y) & \text{if }x\ge0\\
-u(-x,y) & \text{if }x\le0.
\end{cases}
\]
Since $u(0,y)\equiv0$ for $y\ge0$, it is elementary to find that
$\bar{u}$ solves equation 
\[
\begin{cases}
\De\bar{u}=0 & \text{in }\R_{+}^{2},\\
{\displaystyle \frac{\pa\bar{u}}{\pa\nu}}=|\bar{u}|^{p-1}\bar{u} & \text{on }\pa\R_{+}^{2}.
\end{cases}
\]
Furthermore, we deduce from above equation that $\bar{u}_{x}$ satisfies
\begin{equation}
\begin{cases}
\De\bar{u}_{x}=0 & \text{in }\R_{+}^{2},\\
\bar{u}_{x}(x,y)=u_{x}(|x|,y)>0 & \text{in }\R_{+}^{2},\\
{\displaystyle \frac{\pa\bar{u}_{x}}{\pa\nu}}=p|\bar{u}|^{p-1}\bar{u}_{x}\ge0 & \text{on }\pa\R_{+}^{2}.
\end{cases}\label{eq: a reduced eq}
\end{equation}

Applying Proposition \ref{prop: Proposition of 6.2} to Eq. (\ref{eq: a reduced eq})
gives that $\bar{u}_{x}\equiv C$ in $\R_{+}^{2}$ for some constant
$C>0$. Since $\bar{u}(0,y)\equiv0$ for $y\ge0$, we derive that
$\bar{u}(x,y)=Cx$ for all $(x,y)\in\R_{+}^{2}.$ But then, it follows
that $\pa_{\nu}u\equiv0\neq u^{p}$ on $\{x>0,y=0\}$. We reach a
contradiction. The proof of Theorem \ref{thm: Main reuslt} is complete.
\end{proof}

Now we prove Lemma \ref{lem: monotonicity}. We will employ the method
of moving spheres (see Li, Zhang and Zhu \cite{LiYY-2004,Li-Zhang-2003,Li-Zhu-1995}),
a variant of the method of moving planes invented by the Soviet mathematician
Alexanderov in the early 1950s, and later further developed by Serrin
\cite{Serrin-1971}, Gidas et al. \cite{Gidas-Ni-Nirenberg-1979},
Caffarelli et al. \cite{Caffarelli et al. 1989}, Li \cite{LiCM-1996},
Chen and Li \cite{Chen-Li-1991,Chen-Li-1997}, Chang and Yang \cite{Chang-Yang-1997},
Chen et al. \cite{Chen-Li-Ou-2006} and many others. We also make
use of the idea of narrow domains from Berestycki and Nirenberg \cite{Berestycki-Nirenberg-1988}.

\begin{proof}[Proof of Lemma \ref{lem: monotonicity}] First we introduce
some notations for convenience. Denote the point in the plane by $z=(x,y)\in\R^{2}$.
Let $\la,R\in(0,\wq)$, $\la>R$, be arbitrary positive constants
and write $z_{R}=(-R,0)$. For any positive solution $u$ of Eq. (\ref{eq: Cabre-Tan-1D}),
define the function $u_{R,\la}:\Om_{R,\la}\to[0,\wq)$ by 
\begin{eqnarray*}
u_{R,\la}(z)=u\left(z_{R}+\frac{\la^{2}(z-z_{R})}{|z-z_{R}|^{2}}\right) &  & \text{for }z\in\Om_{R,\la},
\end{eqnarray*}
where $\Om_{R,\la}$ is the bounded domain given by 
\[
\Om_{R,\la}=B_{\la}(z_{R})\cap\R_{++}^{2}.
\]

Since $u$ solves Eq. (\ref{eq: Cabre-Tan-1D}), a direct calculation
shows that $u_{R,\la}$ satisfies 
\begin{equation}
\begin{cases}
\De u_{R,\la}=0 & \text{in }\Om_{R,\la},\\
u_{R,\la}>0 & \text{in }\Om_{R,\la},\\
u_{R,\la}=u & \text{on }\pa\Om_{R,\la}\cap\R_{++}^{2},\\
{\displaystyle \frac{\pa u_{R,\la}}{\pa\nu}}=\left(\frac{\la}{|z-z_{R}|}\right)^{2}u_{R,\la}^{p}(z) & \text{on }\pa\Om_{R,\la}\cap\{x>0,y=0\}.
\end{cases}\label{eq: transformed eq}
\end{equation}
Our aim is to show that 
\begin{eqnarray}
u(z)<u_{R,\la}(z) &  & \text{in }\Om_{R,\la}\label{eq: comparison}
\end{eqnarray}
for all $\la,R\in(0,\wq)$ with $\la>R$. 

Let $R>0$ be fixed. First we show that (\ref{eq: comparison}) holds
when $\la-R>0$ is sufficiently small. To this end, set $w_{\la}(z)=u(z)-u_{R,\la}(z)$
for $z\in\Om_{R,\la}$. We have that 
\begin{equation}
\begin{cases}
\De w_{\la}=0 & \text{in }\Om_{R,\la},\\
w_{\la}=0 & \text{on }\pa\Om_{R,\la}\cap\R_{++}^{2},\\
w_{\la}<0 & \text{on }\pa\Om_{R,\la}\cap\{x=0,y>0\},\\
{\displaystyle \frac{\pa w_{\la}}{\pa\nu}}=u^{p}-\left(\frac{\la}{|z-z_{R}|}\right)^{2}u_{R,\la}^{p} & \text{on }\pa\Om_{R,\la}\cap\{x>0,y=0\}.
\end{cases}\label{eq: Eq of comparison}
\end{equation}
Multiply Eq. (\ref{eq: Eq of comparison}) by $w_{\la}^{+}\equiv\max(w_{\la},0)$
and integrate by parts. We deduce that 
\[
\int_{\Om_{R,\la}}|\na w_{\la}^{+}|^{2}=\int_{\pa\Om_{R,\la}\cap\{x>0,y=0\}}w_{\la}^{+}\left(u^{p}-\left(\frac{\la}{|z-z_{R}|}\right)^{2}u_{R,\la}^{p}\right).
\]
Denote 
\[
A_{\la}=\left\{ z\in\pa\Om_{R,\la}\cap\{x>0,y=0\}:w_{\la}(z)>0\right\} .
\]
Since $\la>|z-z_{R}|$ on $\pa\Om_{R,\la}\cap\{x>0,y=0\}$ and $p\ge1$,
we have that 
\[
\int_{\pa\Om_{R,\la}\cap\{x>0,y=0\}}w_{\la}^{+}\left(u^{p}-\left(\frac{\la}{|z-z_{R}|}\right)^{2}u_{R,\la}^{p}\right)\le\int_{A_{\la}}pu^{p-1}\left(w_{\la}^{+}\right)^{2}.
\]
By the local boundedness of $u(x,0)$ for $x>0$, we have that 
\[
\int_{A_{\la}}pu^{p-1}\left(w_{\la}^{+}\right)^{2}\le p\sup_{x\in A_{\la}}u^{p-1}(x,0)\int_{A_{\la}}\left(w_{\la}^{+}\right)^{2}.
\]
Hence combining above estimates together with H\"older's inequality
gives that
\[
\int_{\Om_{R,\la}}|\na w_{\la}^{+}|^{2}\le p\left(\sup_{x\in A_{\la}}u^{p-1}(x,0)\right)|A_{\la}|^{1-\frac{2}{q}}\left(\int_{\Om_{R,\la}}\left(w_{\la}^{+}\right)^{q}\right)^{\frac{2}{q}},
\]
where $2<q<\wq$ is a fixed number. Note that $w_{\la}^{+}=0$ on
$\pa\Om_{R,\la}\cap\{y>0\}$. Applying Lemma \ref{lem: Sobolev-Poincare inequality}
with $\Om=\Om_{R,\la}$, we deduce that 
\begin{equation}
\int_{\Om_{R,\la}}|\na w_{\la}^{+}|^{2}\le C_{p,q}\left(\sup_{x\in A_{\la}}u^{p-1}(x,0)\right)|\Om_{R,\la}|^{\frac{2}{q}}|A_{\la}|^{1-\frac{2}{q}}\int_{\Om_{R,\la}}|\na w_{\la}^{+}|^{2},\label{eq: Self-vanishing ineq.}
\end{equation}
where $C_{p,q}>0$ is a constant depending only on $p$ and $q$. 

Note that $|A_{\la}|\le\la-R$. Thus, it is easy to infer from inequality
(\ref{eq: Self-vanishing ineq.}) that (\ref{eq: comparison}) holds
when $\la-R>0$ is sufficiently small. 

Next we show that for any fixed $R>0$, (\ref{eq: comparison}) holds
for all $\la\in(R,\wq)$. To this end, set 
\[
\bar{\la}(R)=\{\mu\in(R,\infty):(\ref{eq: comparison})\text{ holds for all }R<\la<\mu.\}
\]
We claim that $\bar{\la}(R)=\wq$. Argue by contradiction. Suppose
that $\bar{\la}(R)<\wq$ holds. Then by continuity, we have that $u\le u_{R,\bar{\la}(R)}$
in $\Om_{R,\bar{\la}(R)}$. Since $u<u_{R,\bar{\la}(R)}$ on $\pa\Om_{R,\bar{\la}(R)}\cap\{x=0,y>0\}$,
we deduce that $u<u_{R,\bar{\la}(R)}$ in $\Om_{R,\bar{\la}(R)}$
by the strong maximum principle. Therefore we infer that 
\begin{eqnarray*}
|A_{\la}|\to0 &  & \text{as }\la\downarrow\bar{\la}(R).
\end{eqnarray*}
Thus, there exists a sufficiently small number $\de>0$, such that
\[
C_{p,q}\left(\sup_{0<x<\bar{\la}(R)+\de}u^{p-1}(x,0)\right)\left|\Om_{R,\bar{\la}(R)+\de}\right|^{\frac{2}{q}}|A_{\la}|^{1-\frac{2}{q}}<\frac{1}{2}
\]
for all $\la\in(\bar{\la}(R),\bar{\la}(R)+\de)$. Then combining above
estimate together with inequality (\ref{eq: Self-vanishing ineq.})
yields that $u\le u_{R,\la}$ in $\Om_{R,\la}$ for all $\la\in(\bar{\la}(R),\bar{\la}(R)+\de)$.
This is against the choice of $\bar{\la}(R)$. Hence we conclude that
$\bar{\la}(R)=\wq$. In this way, we show that for any fixed $R>0$,
(\ref{eq: comparison}) holds for all $\la\in(R,\wq)$.

Now we can finish the proof of Lemma \ref{lem: monotonicity}. Let
$(x_{1},y_{0})$ and $(x_{2},y_{0})$, $0<x_{1}<x_{2}$, be two arbitrary
points in $\R_{++}^{2}.$ Then for all $R>0$ sufficiently large,
we have $(x_{1},y)\in B_{R+a}(z_{R})\cap\R_{++}^{2}$, where $a=(x_{1}+x_{2})/2$.
Then applying (\ref{eq: comparison}) with $\la=R+a$ gives that 
\[
u(x_{1},y_{0})<u_{R,R+a}(x_{1},y_{0})
\]
for all $R>0$ sufficiently large. Letting $R\to\wq$ in the above
inequality yields that 
\[
u(x_{1},y_{0})\le u(2a-x_{1},y_{0})=u(x_{2},y_{0}).
\]
This shows that $u$ is monotone increasing in the $x$-direction,
that is, $u_{x}\ge0$ in $\R_{++}^{2}$. 

To derive the strict inequality in Lemma \ref{lem: monotonicity},
we note that $u_{x}$ is also a harmonic function in $\R_{+}^{2}$
and $\pa_{\nu}u_{x}=pu^{p-1}u_{x}\ge0$ on $\{x>0,y=0\}$. Hence it
follows from the strong maximum principle that $u_{x}>0$ in $\R_{++}^{2}$,
and from the Hopf lemma that $u_{x}(0,y)>0$ on $\pa\R_{++}^{2}$.
The proof of Lemma \ref{lem: monotonicity} is complete. \end{proof}

\section{An extension}

Recall that we mentioned the quite general nonexistence result of
Cabr\'e and Tan \cite[Proposition 6.4]{Cabre-Tan-2010} in the introduction
part. It states as follows. 

\begin{proposition} \label{prop: Proposition of 6.4} Assume that
$f$ is a $C^{1,\al}$ function for some $\al\in(0,1)$, such that
$f>0$ in $(0,\wq)$ and $f(0)=0$. Let $C$ be a positive constant.
Then there is no bounded solution of the problem 
\begin{equation}
\begin{cases}
\De u=0, & \text{in }\R_{++}^{2},\\
0<u(x,y)\le C, & \text{in }\R_{++}^{2},\\
u(0,y)=0 & \text{on }\{x=0,y\ge0\},\\
{\displaystyle \frac{\pa u}{\pa\nu}}=f(u) & \text{on }\{x>0,y=0\}.
\end{cases}\label{eq: generalization-1}
\end{equation}
 \end{proposition}

In this section, we give an extension of Theorem \ref{thm: Main reuslt}
in the case $n=1$, which can be seen as an analogue of Proposition
\ref{prop: Proposition of 6.4}.

\begin{theorem} \label{thm: an extension}Assume that $f:[0,\wq)\to[0,\wq)$
is a nondecreasing $C^{1}$ function with $f(0)=0$, and that $u$
is a nonnegative classical solution to the problem
\begin{equation}
\begin{cases}
\De u=0, & \text{in }\R_{++}^{2},\\
u(x,y)\ge0, & \text{in }\R_{++}^{2},\\
u(0,y)=0 & \text{on }\{x=0,y\ge0\},\\
{\displaystyle \frac{\pa u}{\pa\nu}}=f(u) & \text{on }\{x>0,y=0\}.
\end{cases}\label{eq: generalization}
\end{equation}
Then there exists a constant $C\ge0$ such that 
\begin{eqnarray*}
u(x,y)=Cx &  & \text{for }(x,y)\in\R_{++}^{2}.
\end{eqnarray*}
 \end{theorem} 
\begin{proof}
We only give a sketch of the proof. Let $u$ be a positive solution
to Eq. (\ref{eq: generalization-1}). 

First we show that $u$ is nondecreasing in the $x$-direction. Define
$\Om_{R,\la}$ and $u_{R,\la}:\Om_{R,\la}\to[0,\wq)$ as in the proof
of Lemma \ref{lem: monotonicity}. It is elementary to derive that
\[
\begin{cases}
\De u_{R,\la}=0 & \text{in }\Om_{R,\la},\\
u_{R,\la}\ge0 & \text{in }\Om_{R,\la},\\
u_{R,\la}=u & \text{on }\pa\Om_{R,\la}\cap\R_{++}^{2},\\
{\displaystyle \frac{\pa u_{R,\la}}{\pa\nu}}=\left(\frac{\la}{|z-z_{R}|}\right)^{2}f\left(u_{R,\la}\right) & \text{on }\pa\Om_{R,\la}\cap\{x>0,y=0\}.
\end{cases}
\]
Then set $w_{\la}=u-u_{R,\la}$ in $\Om_{R,\la}$. Since $f$ is nondecreasing
and continuously differentiable, we deduce that 
\[
\int_{\Om_{R,\la}}|\na w_{\la}^{+}|^{2}\le C_{p,q}\left(\sup_{0\le t\le\|u\|_{\wq,A_{\la}}}f^{\prime}(t)\right)|\Om_{R,\la}|^{\frac{2}{q}}|A_{\la}|^{1-\frac{2}{q}}\int_{\Om_{R,\la}}|\na w_{\la}^{+}|^{2},
\]
where $A_{\la}$ is defined as in the proof of Lemma \ref{lem: monotonicity}.
Above inequality is a counterpart of (\ref{eq: Self-vanishing ineq.}).
Thus we can conclude as in the proof of Lemma \ref{lem: monotonicity}
that $u$ is nondecreasing in the $x$-direction. 

Next, consider the odd extension $\bar{u}$ of $u$ with respect to
$\{x=0,y>0\}$. We deduce that $\bar{u}_{x}$ satisfies 
\[
\begin{cases}
\De\bar{u}_{x}=0 & \text{in }\R_{+}^{2},\\
\bar{u}_{x}(x,y)=u_{x}(|x|,y)\ge0 & \text{in }\R_{+}^{2},\\
{\displaystyle \frac{\pa\bar{u}_{x}}{\pa\nu}}=\bar{f}^{\prime}(\bar{u})\bar{u}_{x}\ge0 & \text{on }\pa\R_{+}^{2},
\end{cases}
\]
where $\bar{f}$ is the odd extension of $f$, that is, $\bar{f}(t)=f(t)$
for $t\ge0$ and $\bar{f}(t)=-f(-t)$ for $t<0$. Now Theorem \ref{thm: an extension}
follows from Proposition \ref{prop: Proposition of 6.2} easily. 
\end{proof}
In the spirit of Theorem \ref{thm: Cabre-Tan-2010}, we have the following
application of Theorem \ref{thm: an extension}. 

\begin{corollary} Assume that $f:[0,\wq)\to[0,\wq)$ is a nondecreasing
$C^{1}$ function with $f(0)=0$. Then, there exists no bounded solution
to the problem 
\[
\begin{cases}
A_{1/2}u=f(u) & \text{in }\R_{+}=\{x>0\},\\
u>0 & \text{in }\R_{+},\\
u(0)=0,
\end{cases}
\]
where $A_{1/2}$ is the square root of the Laplacian in $(0,\wq)$
with zero Dirichlet boundary conditions at $x=0$. \end{corollary}

\emph{Acknowledgment. }
The author is financially supported by the Academy of Finland, project
259224.

\end{document}